\documentclass[twocolumn]{autart}    
                                     
\usepackage[T1]{fontenc}
\usepackage[utf8]{inputenc}
\usepackage{natbib}
\usepackage{url}
\usepackage{amsmath}
\allowdisplaybreaks
\usepackage{booktabs}      
\usepackage{amssymb}
\usepackage{enumitem}
\usepackage{bm}
\usepackage[dvipsnames]{xcolor}
\usepackage{color}\definecolor{RED}{rgb}{1,0,0}\definecolor{BLUE}{rgb}{0,0,1}\definecolor{GRAY}{rgb}{0.4,0.4,0.4}
\usepackage{algorithm}
\usepackage{mathtools}
\usepackage{etoolbox}
\mathtoolsset{showonlyrefs}
 
\usepackage[caption=false,font=footnotesize,subrefformat=parens,labelformat=parens]{subfig}

\usepackage[colorlinks=true,citecolor = NavyBlue, linkcolor= NavyBlue]{hyperref}
\edef\endfrontmatter{\unexpanded\expandafter{\endfrontmatter} 
  \noexpand\endNoHyper}

\usepackage{wrapfig}
\let\classAND\AND
\let\AND\relax
\usepackage[noend]{algorithmic}

\let\AND\classAND
\AtBeginEnvironment{algorithmic}{\let\AND\algoAND}

\floatname{algorithm}{Algorithm}

\newcommand{\minimize}{\operatorname*{minimize}}

\newcommand{\st}{\operatorname*{subject\ to}}

\newcommand{\ra}{\rightarrow}

\newcommand{\mathletter}[1]{%
\expandafter\newcommand\csname b#1\endcsname{\mathbb #1}
\expandafter\newcommand\csname c#1\endcsname{\mathcal #1}
\expandafter\newcommand\csname f#1\endcsname{\mathfrak #1}
\expandafter\newcommand\csname til#1\endcsname{\widetilde #1}
\expandafter\newcommand\csname ha#1\endcsname{\widehat #1}
\expandafter\newcommand\csname bf#1\endcsname{\bf #1}
\expandafter\newcommand\csname s#1\endcsname{\mathsf #1}
}%
\newcommand{\mydefa}[1]{\ifx#1\mydefa\else\mathletter#1\expandafter\mydefa\fi}
\mydefa ABCDEFGHIJKLMNOPQRSTUVWXYZ\mydefa

\newcommand{\mathletterl}[1]{%
\expandafter\providecommand\csname v#1\endcsname{\vec{#1}}
}%
\newcommand{\mydefb}[1]{\ifx#1\mydefb\else\mathletterl#1\expandafter\mydefb\fi}
\mydefb abcdefghijklmnopqrstuvwxyz\mydefb

\newcommand{\bea}{\begin{equation}\begin{alignedat}{-1}}
\newcommand{\ena}{\end{alignedat}\end{equation}}
\newcommand{\bee}{\begin{equation}}
\newcommand{\ene}{\end{equation}}

\renewcommand{\vec}[1]{\mathbf{#1}}
\providecommand{\diff}[1]{{\protect{#1}}}

\newtheorem{theo}{Theorem}
\newtheorem{lemma}{Lemma}

\newtheorem{remark}{Remark}

\makeatletter
\DeclareRobustCommand{\qed}{%
  \ifmmode 
  \else \leavevmode\unskip\penalty9999 \hbox{}\nobreak\hfill
  \fi
  \quad\hbox{\qedsymbol}}
\newcommand{\openbox}{\leavevmode
  \hbox to.77778em{%
  \hfil\vrule
  \vbox to.675em{\hrule width.6em\vfil\hrule}%
  \vrule\hfil}}
\newcommand{\qedsymbol}{\openbox}
\newenvironment{proof}[1][\proofname]{\par
  \normalfont
  \topsep6\p@\@plus6\p@ \trivlist
  \item[\hskip\labelsep\itshape
    #1.]\ignorespaces
}{%
  \qed\endtrivlist
}
\newcommand{\proofname}{Proof}
\makeatother

\newcommand{\T}{\mathsf{T}}

\begin{document}

\begin{frontmatter}

\title{Distributed Adaptive Newton Methods with Global Superlinear Convergence\thanksref{footnoteinfo}} 

\thanks[footnoteinfo]{This work was supported by the National Natural Science Foundation of China under grant no. 62033006, a grant from the Guoqiang Institute, Tsinghua University, and in part by the ARL under cooperative agreement W911NF-17-2-0196. The material in this paper was partially presented at the 59th IEEE Conference on Decision and Control.}
\author[thu]{Jiaqi~Zhang}\ead{zjq16@mails.tsinghua.edu.cn},    
\author[thu]{Keyou~You\corauthref{cor}}\ead{youky@tsinghua.edu.cn},               
\author[uiuc]{Tamer~Ba\c{s}ar}\ead{basar1@illinois.edu}  

\corauth[cor]{Corresponding author}

\address[thu]{Department of Automation, and BNRist, Tsinghua University, Beijing 100084, China.}  
\address[uiuc]{Coordinated Science Laboratory, University of Illinois at Urbana-Champaign, Urbana, IL 61801 USA.}             


\begin{abstract}                          
    This paper considers the distributed optimization problem  where each node of a peer-to-peer network minimizes a finite sum of objective functions  by communicating with its neighboring nodes. In sharp contrast to the existing literature where the fastest distributed algorithms converge either with a global linear or a local superlinear rate, we propose a distributed adaptive Newton (DAN) algorithm with a {\em global quadratic} convergence rate. Our key idea lies in the design of a finite-time set-consensus method with Polyak's adaptive stepsize. Moreover, we introduce a low-rank matrix approximation (LA) technique to compress the innovation of Hessian matrix so that each node only needs to transmit message of dimension $\cO(p)$ (where $p$ is the dimension of decision vectors) per iteration, which is essentially the same as that of first-order methods. Nevertheless, the resulting DAN-LA converges to an optimal solution with a \emph{global superlinear} rate. Numerical experiments on logistic regression problems are conducted to validate their advantages over existing methods. 
\end{abstract}

\end{frontmatter}

\section{Introduction}\label{sec1}

Distributed optimization entails solving the following problem over a peer-to-peer network system
\bea\label{obj2}
&\minimize_{\vx_1,\dots,\vx_n}&&F(\vx_1,\dots,\vx_n)\triangleq \sum_{i=1}^{n}f_i(\vx_i)\\
&\st &&\vx_1=\cdots=\vx_n\in\bR^p
\ena
where each node $i$ privately holds a local objective function $f_i$ and updates its decision vector $\vx_i$ via communicating with its neighboring nodes. Our goal is to design efficient distributed algorithms to find an optimal solution  of \eqref{obj2}. Many efforts have been devoted along this line, see e.g., \citet{nedic2017achieving,xin2019distributed2,scaman2017optimal,qu2019accelerated}. It is known that the fastest rate for first-order methods is linear \citep{nesterov2018lectures}, and second-order methods are unavoidable for the superlinear convergence. 

Unfortunately, the existing distributed methods cannot achieve \emph{global superlinear} convergence rates. In contrast, our paper proposes two Newton-based distributed algorithms with global quadratic and superlinear convergence rates, respectively. A direct comparison with the existing literature can be found in Table \ref{tab1}. For second-order methods in \citet{mokhtari2017network,mansoori2019fast,mansoori2017superlinearly,tutunov2019distributed} and \citet{zargham2014accelerated}, an \emph{inexact} penalization reformulation as an unconstrained optimization problem is adopted to solve  \eqref{obj2}. \diff{Their superlinear rate is restricted to a limited region, which does not include an optimal solution, and then reduced to a linear rate. Even though the inexact issue has been resolved in \citet{eisen2017decentralized,eisen2018primal} and \citet{varagnolo2016newton}, their convergence rates are still linear.}
 
Under a master-slave network configuration, some distributed quasi-Newton methods have been proposed in \citet{shamir2014communication,wang2018giant,zhang2015disco}, and \citet{soori2019dave}.  Though the master node of this setting uses all the information from the slave nodes, \diff{the algorithms in \citet{shamir2014communication,zhang2015disco} and \citet{wang2018giant} still converge linearly,} and only \citet{soori2019dave} achieves a \emph{local} superlinear convergence rate if the starting point is sufficiently close to an optimal solution. 

To achieve a global superlinear convergence rate,  two bottlenecks have to be resolved. The first is the use of the linear consensus algorithm whose convergence rate is at most linear. Thus, the resulting second-order methods are still constrained to the linear convergence rate. Similar to this work,  finite-time consensus methods have been adopted in \citet{qu2019finite} to achieve local quadratic convergence by using the pure Newton direction.

The second bottleneck lies in the design of a backtracking line search method to distributedly tune stepsizes of the second-order methods. Even though the pure Newton direction with a unit stepsize is key to the superlinear convergence, it may lead to divergence if the algorithm is far from an optimal solution  \citep{boyd2004convex}. In fact, the line search in the early stage is vital to the global convergence of second-order methods and usually requires nodes to evaluate the objective function of $F$ multiple times in a single iteration, which is not possible in our distributed setting as each node $i$ can only access its individual objective function $f_i$. To circumvent this problem, \diff{local line search methods have been designed in \citet{zargham2012distributed} and \citet{jadbabaie2009distributed} for network flow problems, and fixed stepsizes have been used in \citet{mokhtari2017network,mansoori2019fast,mansoori2017superlinearly,tutunov2019distributed,mokhtari2016decentralized,eisen2018primal,varagnolo2016newton}. However, they are too conservative to achieve superlinear convergence}.


In this work, we  resolve the above issues and design distributed algorithms with global superlinear convergence rates.  \diff{Building on the distributed flooding (DF) algorithm \citep{liu2007distributed,dias2013cooperative,li2017convergence}}, we propose a distributed selective flooding (DSF) algorithm to achieve exact consensus in a finite number of local communication rounds and improves the DF with provably lower communication complexity. Precisely, nodes in the DSF achieve consensus with at most $n-1$ rounds of local communication rather than $n+d_\cG-1$ in the DF, where $d_\cG$ is the diameter of the network with $n$ nodes.  By running the DSF at each iteration, the consensus achievement is no longer a bottleneck for superlinear convergence. 

Then, we leverage the adaptive stepsize in  \citet{polyak2019new}, which is especially amenable to the distributed setting, and  propose the Distributed Adaptive Newton algorithm (DAN)  to maintain the quadratic convergence rate from any starting point. In DAN, each node $i$ computes the gradient and Hessian of $f_i$, and then transmits to its neighbors to perform the DSF, after which the Newton direction and the adaptive stepsize are computed for updates.  Note that \citet{polyak2019new} does not address any distributed issue. 

To explicitly reduce the communication cost per iteration, we further propose a Low-rank Approximation (LA) technique to design a communication-efficient DAN-LA. The idea builds on the observation that the rank-1 matrix approximation of a symmetric matrix $W\in\bR^{p\times p}$ with respect to (w.r.t.) the spectral norm can be expressed as $s\cdot \vw\vw^\T,$ where $s\in\{-1,1\}$ and $\vw\in\bR^p$. As first-order methods \citep{shi2015extra,nedic2017achieving,scaman2017optimal,qu2019accelerated,zhang2019asynchronous}, using such an approximation of $W$ only requires transmitting a vector of dimension $\cO(p)$. Then, we apply this idea to the ``innovation'' of Hessian, which is defined as the difference between the true Hessian and its estimate. \diff{Noticeably different from the quasi-Newton methods (e.g., BFGS \citep{nesterov2018lectures,nocedal2006numerical}), here we can control the approximation error to co-design stepsizes to ensure global superlinear convergence}. Note that the truncation method in \citet{mokhtari2017network} cannot achieve a global superlinear convergence rate. Finally, we test DAN and DAN-LA on a logistic regression problem over the Covertype dataset. The results confirm the advantages of DAN and DAN-LA and validate our theoretical results.

The rest of this paper is organized as follows. Section \ref{sec2} formulates the problem. Section \ref{sec3} introduces the finite-time DSF with theoretical guarantees. Section \ref{sec4} first reviews the Newton method and Polyak's adaptive stepsize, and then proposes DAN. Section \ref{sec5} presents DAN-LA with low-rank matrix approximation methods. We test the two algorithms in Section \ref{sec6} on a logistic regression problem, and conclude the paper in Section \ref{sec7}. A conference version of this paper appears in \citet{zhang2020achieving}, where none of the proofs have been included and the DSF method in Section \ref{sec3c} does not appear. 

\def\arraystretch{1.2}
\begin{table*}[!t]
	\centering
	\caption{A brief summary of distributed second-order methods.}
	\label{tab1}
	\begin{minipage}{\linewidth}
		\centering
	\begin{tabular}{lll}
		\toprule
		Literature                                                                                       & Convergence rates & Main features                 \\
		\midrule
		\begin{tabular}[c]{@{}l@{}}\citet{mansoori2017superlinearly,mansoori2019fast}\\\citet{tutunov2019distributed,zargham2014accelerated}\\ \citet{mokhtari2017network}\end{tabular}  & Linear\footnote{\diff{Superlinear rate may happen in a region that does not include the optimal solution, and then the rate reduces to linear.}}, \emph{inexact}\footnote{Cannot converge to an exact optimal solution}  & \begin{tabular}[c]{@{}l@{}} Approximate the Hessian of a penalty  \\ function of \eqref{obj2} with truncated Taylor series \end{tabular} \\
		\midrule
		\citet{mokhtari2016decentralized,eisen2017decentralized,eisen2018primal}                                             & Linear           & Primal-dual methods                       \\
		\citet{varagnolo2016newton,zhang2020newton}                                                   & Linear           & Asymptotic average consensus                        \\
		\citet{shamir2014communication,wang2018giant}                                                 & Linear           & Master-slave networks   \\
		\citet{soori2019dave}                                                                         & Superlinear, \emph{local}\textsuperscript{\ref{footnote1}}     & Master-slave networks                        \\
		\citet{qu2019finite}                                                                          & Quadratic, \emph{local}\footnote{\label{footnote1}The starting point should be sufficiently close to the optimal solution} & Finite-time average consensus                        \\
		\midrule
		\textbf{DAN} of this work                                                                                        & Quadratic        & Adaptive stepsize and finite-time consensus                       \\
		\textbf{DAN-LA}  of this work                                                                                         & Superlinear      & Low-rank approximation to compress Hessians \\
		\bottomrule
	\end{tabular}
\end{minipage}
\end{table*}

{\bf Notation:}~$\vx^\T$ denotes the transpose of $\vx$. $\|\cdot\|$ denotes the $l_2$-norm for vectors and the spectral norm for matrices. $|\cX|$ denotes the cardinality of set $\cX$. $\nabla f$ and $\nabla^2 f$ denote the gradient and Hessian of $f$, respectively. $\cO(\cdot)$ denotes the big-O notation. $A\succeq B$ means $A-B$ is non-negative definite. The ceiling function $\lceil x\rceil$ returns the least integer greater than or equal to $x$. A sequence $\{e_k\}$ converges to 0 with a $Q$-linear rate if $\lim_{k\ra\infty}{|e_{k+1}|}/{|e_k|}=\gamma$ for some $\gamma\in(0,1)$, a $Q$-superlinear rate if $\lim_{k\ra\infty} {|e_{k+1}|}/{|e_k|}=0$, and a $Q$-quadratic rate if $\lim_{k\ra\infty} {|e_{k+1}|}/{|e_k^2|}=\gamma$ for some $\gamma>0$. A sequence $\{\varepsilon_k\}$ converges $R$-linearly (superlinearly, quadratically) if there exists a $Q$-linearly (superlinearly, quadratically) convergent sequence $\{e_k\}$ such that $|\varepsilon_k|\leq |e_k|,\forall k$. In this paper, the convergence rates  are always in the sense of $R$-rate.

\section{Problem Formulation}\label{sec2}

The communication network among  nodes is modeled by a directed network $\cG=(\cV,\cE)$, where $\cV=\{1,\cdots,n\}$ denotes the set of nodes and $\cE\subseteq\cV\times\cV$ is the set of edges with $(i,j)\in\cE$ if and only if node $i$ can directly send messages to $j$. Let $\cN_\text{in}^i=\{j|(j,i)\in\cE\}$ denote the set of in-neighbors and $\cN_\text{out}^i=\{j|(i,j)\in\cE\}$ is the set of out-neighbors of node $i$. If $\cG$ is undirected, i.e., $(i,j)\in\cE$ implies $(j,i)\in\cE$,  we simply denote by $\cN_i$ the neighbors of node $i$. A path from node $i$ to node $j$ is a sequence of consecutively directed edges from $i$ to $j$. We say $\cG$ is \emph{strongly connected} if there exists a directed path between any pair of nodes. The distance between two nodes in $\cG$ is the number of edges in a shortest path connecting them, and the diameter $d_\cG$ of $\cG$ is the largest distance between any pair of nodes. A \emph{tree} is an undirected network that is strongly connected without any cycle.

Essentially, \eqref{obj2} is  equivalent to the following optimization problem
\bee\label{obj}
\minimize_{\vx\in\bR^p}\ f(\vx)\triangleq \sum_{i=1}^{n}f_i(\vx).
\ene 
We invoke the following assumptions in the paper.
\begin{assum}\label{assum1}
	\begin{enumerate}
		\renewcommand{\labelenumi}{\rm(\alph{enumi})}
		\item $f$ is twice continuously differentiable and $\mu$-strongly convex, i.e., there exists a positive $\mu>0$ such that $\nabla^2 f(\vx)\succeq\mu I, \forall \vx$.
		\item $f$ has  $L$-Lipschitz continuous Hessian, i.e.,
		      \bee
		      \|\nabla^2 f(\vx)-\nabla^2 f(\vy)\|\leq L\|\vx-\vy\|,\ \forall \vx,\vy.
		      \ene
	\end{enumerate}
\end{assum}
\begin{assum}\label{assum3}
	\begin{enumerate}
		\renewcommand{\labelenumi}{\rm(\alph{enumi})}
		\item The communication network is strongly connected.
		\item Each node has a unique identifier.
	\end{enumerate}
\end{assum}

Assumption \ref{assum1} is standard in the second-order methods \citep{boyd2004convex}, which implies that $f$ has a unique global minimum point $\vx^\star$, i.e., $f(\vx^\star)=\min_\vx f(\vx)$. 
Assumption \ref{assum3} is common to handle directed networks \citep{xie2018distributed,xi2018linear} and is easily satisfied. For example, nodes are generally equipped with network interface cards (NIC) for communication, and each card is assigned to a unique MAC address. The MAC address then naturally serves as an identifier. 

\section{Distributed Finite-time Set-consensus}\label{sec3}

This section first introduces the concept of distributed finite-time set-consensus that plays a crucial role in both DAN and DAN-LA. Following that, we propose a Distributed Selective Flooding (DSF) algorithm that improves the communication efficiency of the existing DF.

\subsection{Distributed Finite-time Set-consensus}

Set-consensus aims to ensure that all nodes reach consensus on a set of values among nodes, which is different from value-consensus to agree on some value. Finite-time consensus achieves the consensus goal only after a \emph{finite}  number of rounds of communication with neighbors. Until now, most distributed algorithms \citep{shi2015extra,nedic2017achieving,scaman2017optimal,qu2019accelerated,zhang2019asynchronous}  to solve \eqref{obj2}  have been built on asymptotic average-consensus methods with linear iterations, which can only attain linear convergence and become a bottleneck for the design of superlinearly convergent  algorithms. To resolve this issue, we design the DSF to achieve set-consensus, which improves the DF in \citet{liu2007distributed,dias2013cooperative,li2017convergence}. 



\subsection{The Distributed Selective Flooding Algorithm}\label{sec3c}


The DSF is given in Algorithm \ref{DSF}, where $\cI_i(t)$ contains the information that node $i$ has received after $t$ communication rounds, and is updated iteratively.  The implementation for undirected and directed networks are slightly different. We first focus on the former case. Let $\cS_i$ be the information to be shared by node $i$, which can be a scalar, a vector, a matrix, or a set. Each node $i$ initializes a set $\cI_i(0)=\{\cS_i\}$. At the $k$-th round, for each neighbor $j$, node $i$ selects an element $e\in\cI_i(k-1)$ \diff{such that (a) node $i$ has not sent $e$ to node $j$, and (b) node $i$ has not received $e$ from node $j$ before. Then, node $i$ sends $e$ to node $j$ (Line \ref{line3:dsf})}. Meanwhile, it receives an element $e_j$ from node $j$ and copies to $\cI_i(k-1)$, i.e., $\cI_i(k)=\cI_i(k-1)\cup\{e_j | j\in\cN_i\}$ (Line \ref{line6:dsf}). 

We show in Theorem \ref{theo1} that each node $i$ of DSF obtains a set containing the information of all nodes  after $n-1$ rounds of communication, i.e., $\cI_i(n-1)=\{\cS_u|u\in\cV\}$. That is, nodes reach set-consensus in finite time. In DSF, each message sent from node $i$ to node $j$ is `new' to node $j$, and hence no messages shall be repeatedly transmitted over any link $(i,j)$. \diff{To this end, unique identifiers are needed to distinguish messages (c.f. Assumption \ref{assum3}(b)), which requires additional $\lceil \log_2(n)\rceil$ bits of memory for each message than iterative consensus methods (e.g. \citet{olfati2004consensus}). On a positive note, however, it is less than 16 bits even for $n\leq 6\times 10^5$}. 

For directed networks, the in-neighbors of a node can be different from its out-neighbors. It is infeasible to check whether an element $e$ has been received from an out-neighbor $j$. Then, the DSF is modified and a node selects $e$ that is not sent to its out-neighbor $j$, i.e., Line 3 is replaced by Line 5 in Algorithm \ref{DSF}. However, this slows down the consensus seeking and duplicate transmissions may happen. In fact, the DSF in this case reduces to the DF, see Remark \ref{remark2} for details.


\begin{algorithm}[!t]
	\caption{The DSF --- from the viewpoint of node $i$}\label{DSF}
	\begin{algorithmic}[1]
		\REQUIRE A message $\cS_i$, which can be a scalar, a vector, a matrix, or a set, and let $\cI_i(0)=\{\cS_i\}$.
		\FOR {$k=1,2,\cdots,n-1$}
		\IF {$\cG$ is undirected}
		\STATE\label{line3:dsf} For each node $j\in\cN_i$, \diff{node $i$ selects an element $e\in\cI_i(k-1)$ such that (a) node $i$ has not sent $e$ to node $j$, and (b) node $i$ has not received $e$ from node $j$ before. Then, node $i$ sends $e$ to node $j$.}
		\ELSIF {$\cG$ is directed}
		\STATE For each node $j\in\cN_{\text{out}}^{i}$, node $i$ selects an element $e\in\cI_i(k-1)$ that has not been sent to node $j$ \diff{from node $i$}, and sends $e$ to node $j$.
		\ENDIF
		\STATE\label{line6:dsf} Node $i$ receives an element $e_j$ from each neighbor $j$ and copies to $\cI_i(k-1)$, i.e.,
		\bee
		\cI_i(k)=\cI_i(k-1)\cup\{e_j | j\in\cN_i\}.
		\ene
		\ENDFOR
		\ENSURE Each node obtains a set $\cI_i(n-1)=\{\cS_u|u\in\cV\},\forall i\in\cV$ which contains messages of all nodes.
	\end{algorithmic}
\end{algorithm}

\begin{theo}\label{theo1}
	Suppose Assumption \ref{assum3} holds. If $\cG$ is a tree, then each node in Algorithm \ref{DSF} obtains a set containing all nodes' messages after $n-1$ rounds of local communication with its neighbors, i.e., $\cI_i(n-1)=\{\cS_u|u\in\cV\},\forall i\in\cV$. If $\cG$ is directed, then the maximum number of communication rounds is $n+d_\cG-1$.
\end{theo}


The proof is relegated to Appendix \ref{apx}. Although Theorem \ref{theo1} only considers undirected trees, the result also holds for other types of undirected networks where we can first construct a spanning tree in a distributed manner and then apply the DSF on the selected tree.

\diff{
\begin{remark}\label{remark3}
	A spanning tree can be constructed in a distributed way via a simple breadth-first search, see e.g. \citet[Section 4.2]{lynch1996distributed}, where $\cO(d_{\cG})$ communication rounds are needed and the total number of communication bits of all nodes is at most $\cO(|\cE|\log(n))$ for undirected graphs. Since we focus on static networks,  the tree only needs to be constructed  once. The DSF is also resilient against dropping some edges if the resulting time-varying graphs always contain a fixed spanning tree.  As Theorem 1 dictates, the essential requirement of the graph lies in the existence of a spanning tree.  Whether the network conditions can be relaxed  to those supporting iterative consensus \citep{olfati2004consensus} is not addressed in this work.
\end{remark}
}


\begin{remark}[Optimality for trees]\label{remark1}
	Since there exists at least one node with only one neighbor in a tree, such a node has to receive $n-1$ transmissions from its neighbor to achieve set-consensus. Thus, the set-consensus needs at least $n-1$ number of transmissions, showing the optimality of Theorem \ref{theo1}.
\end{remark}

\begin{remark}[Improvement over the DF]\label{remark2}
	The DF is designed for directed networks and is equivalent to our DSF for directed networks. However, the DSF improves the DF on undirected networks by eliminating the duplicate transmissions of the same message over an edge. Since the diameter of a line network is $d_\cG=n-1$, the DSF only requires a half number of communications than that of the DF.  Despite the widespread use of the DF,
	its minimum number of communication rounds has only recently been established \citep[Theorem 1]{oliva2017distributed}. 
\end{remark}


\section{The Distributed Adaptive Newton Method}\label{sec4}
We first review the centralized Newton method with a backtracking line search, and reveals its difficulty in the distributed setting. Then, we introduce an adaptive Newton method \citep{polyak2019new}, which does not require any line search while maintaining a quadratic convergence rate. Finally, we integrate the adaptive Newton method with the DSF to propose the Distributed Adaptive Newton method (DAN) and provide its convergence result.

\subsection{The  Newton Method}\label{sec4a}

The  Newton method has the following update rule
\bee\label{newton}
\vx_{k+1}=\vx_{k}-\alpha_k(\nabla^2 f(\vx_k))^{-1}\nabla f(\vx_k).
\ene
If the stepsize $\alpha_k=1$, then \eqref{newton} is the \emph{pure} Newton method, which converges quadratically only when the starting point $\vx_0$ is sufficiently close to $\vx^\star$, i.e., $\|\vx_0-\vx^\star\|\leq\mu^2/L$ \citep{nesterov2018lectures}. If $\vx_0$ is far from $\vx^\star$, the pure Newton method may diverge.


To ensure global convergence,  $\alpha_k$ in \eqref{newton} is usually determined by a backtracking line search method. For example, the popular Armijo rule chooses $\alpha_k=\beta^l$ where $\beta\in(0,1)$ and $l$ is the smallest nonnegative integer satisfying
\bee\label{eq_armijo}
f(\vx_k+\beta^l\vd_k)-f(\vx_k)\leq \gamma\beta^l\nabla f(\vx_k)^\T\vd_k
\ene
where $\gamma\in(0,1)$ and $\vd_k=(\nabla^2 f(\vx_k))^{-1}\nabla f(\vx_k)$ is the Newton direction. Even though this often promises a global superlinear convergence rate,  it needs at least $\cO((l+1)(n-1))$ communication rounds in the distributed setting per iteration, leading to a huge communication overhead.  If $\vx_k$ is far away from an optimal point $\vx^\star$, then $l$ in \eqref{eq_armijo} is often large, and thus the line search is not suitable for distributed implementation.


\subsection{The Polyak's Adaptive Newton Method}

In \citet{polyak2019new}, a  damped Newton method  with  adaptive stepsizes has been proposed  to solve nonlinear equations, and has a basic form as follows
\bea\label{ada_nw}
\alpha_k&=\min\Big\{1,\frac{\mu^2}{L\|\nabla f(\vx_k)\|}\Big\}\\
\vx_{k+1}&=\vx_{k}-\alpha_k(\nabla^2 f(\vx_k))^{-1}\nabla f(\vx_k)
\ena
where $\mu$ and $L$ are  defined in Assumption \ref{assum1}.

The idea in \eqref{ada_nw} is very natural. When $\vx_k$ is far from the optimal point $\vx^\star$, the algorithm is in the damped Newton phase with the stepsize inversely proportional to the size of $\nabla f(\vx_k)$. When $\vx_k$ is close to  $\vx^\star$, then \eqref{ada_nw} switches to the pure Newton method. Even though it does not involve any line search, it still achieves a global quadratic convergence rate. Particularly, the number of iterations to achieve $\|\vx_k-\vx^\star\|\leq\epsilon$ is $O\big(\log_2\log_2({1}/{\epsilon})\big)$, which matches the theoretical rate of the Newton method with line search in \citet{boyd2004convex}.

\subsection{The Distributed Adaptive Newton Method}

We use the DSF to extend the above idea to propose DAN in Algorithm \ref{DAN} to solve the distributed optimization problem.  At each iteration, each node $i$ computes the local gradient and Hessian of $f_i$ (Line \ref{line2:dan}), and then runs the DSF to obtain the gradient and Hessian of $f$ (Lines \ref{line3:dan} and \ref{line4:dan}). 
Finally, each node uses the  aggregated gradient and Hessian to perform a Newton step with the adaptive Polyak's  stepsize (Line \ref{line5:dan}). For brevity, all nodes are assumed to start from the same point $\vx_0$, which can  be relaxed by adding a finite-time value-consensus step. 


\begin{algorithm}[!t]
	\caption{The Distributed Adaptive Newton method (DAN) --- from the view of node $i$}\label{DAN}
	\begin{algorithmic}[1]
		\REQUIRE Starting point $\vx_{i}^{0}=\vx_0,\forall i$.
		\FOR {$k=0,1,2,\cdots$}
		\STATE\label{line2:dan} Compute $\vg_k^{(i)}=\nabla f_i(\vx_k^{(i)})$ and $H_k^{(i)}=\nabla^2 f_i(\vx_k^{(i)})$.
		\STATE\label{line3:dan} Obtain $\cS=\{(\vg_k^{(u)},H_k^{(u)}),u\in\cV\}$ by performing the DSF (Algorithm \ref{DSF}) via $n-1$ communication rounds with neighbors.
		\STATE\label{line4:dan} Compute the global gradient and Hessian $\bar\vg_k=\sum_{u=1}^n\vg_k^{(u)}$ and $\bar H_k=\sum_{u=1}^nH_k^{(u)}$.
		\STATE\label{line5:dan} Let $\alpha_{k}^{(i)}=\min\Big\{1,\frac{\mu^2}{L\|\bar\vg_k\|}\Big\}$
		and update
		$$
		\vx_{k+1}^{(i)}=\vx_{k}^{(i)}-\alpha_k^{(i)} (\bar H_{k})^{-1}\bar\vg_{k}.
		$$
		\ENDFOR
	\end{algorithmic}
\end{algorithm}

\begin{theo}\label{theo2}
	Suppose that Assumptions \ref{assum1} and \diff{\ref{assum3}} hold. Then, all sequences $\{\vx_k^{(i)}\},i\in\cV$ in Algorithm \ref{DAN} converge to an optimal point $\vx^\star$, and $\|\nabla f(\vx_k^{(i)})\|$ is monotonically decreasing.  Moreover, let
	\bea\label{eq1_lemma1}
	k_0&=\max\left\{0,\Big\lceil\frac{2L}{\mu^2}\|\nabla f(\vx_0)\|\Big\rceil-2\right\},\\
	\gamma&=\frac{L}{2\mu^2}\|\nabla f(\vx_0)\|-\frac{k_0}{4}\in[0,\frac{1}{2}).
	\ena
	Then, it holds that
	\bee
	\|\nabla f(\vx_k^{(i)})\|\leq\begin{cases}
		\|\nabla f(\vx^0)\|-\frac{\mu^2}{2L}k, & k\leq k_0 \\[4pt]
		\frac{2\mu^2}{L} \gamma^{2^{(k-k_0)}}, & k>k_0
	\end{cases}
	\ene
	and
	\bee
	\|\vx_k^{(i)}-\vx^\star\|\leq\begin{cases}
		\frac{\mu}{L}(k_0-k+\frac{2\gamma}{1-\gamma}),  & k\leq k_0, \\[4pt]
		\frac{2\mu\gamma^{2^{(k-k_0)}}}{L(1-\gamma^{2^{(k-k_0)}})}, & k>k_0.
	\end{cases}
	\ene
\end{theo}

\begin{proof}
	One can  use mathematical induction to easily show that $\vx_k^{(1)}=\dots=\vx_k^{(n)}$ for all $k$, i.e., all nodes' states are identical at any time. Let $\bar\vx_k\triangleq\vx_k^{(1)}=\dots=\vx_k^{(n)}$. Then, each node actually performs the following update
	\bea
	\alpha_k&=\min\Big\{1,\frac{\mu^2}{L\|\nabla f(\bar\vx_k)\|}\Big\}\\
	\bar\vx_{k+1}&=\bar\vx_{k}-\alpha_k(\nabla^2 f(\bar\vx_k))^{-1}\nabla f(\bar\vx_k)
	\ena
	which is exactly Polyak's adaptive Newton method \eqref{ada_nw}. Thus, the result follows from \citet[Theorem 4.1]{polyak2019new}.
\end{proof}

By Theorem \ref{theo2},  $\|\nabla f(\vx_k^{(i)})\|$ and $\|\vx_k^{(i)}-\vx^\star\|$ in each node of DAN decrease by at least a constant positive value at each iteration when $k\leq k_0$, after which the decreasing rate becomes quadratic. Specifically, the number of iterations for $\|\vx_k^{(i)}-\vx^\star\|\leq\epsilon$ in each node is given by
\bee\label{eq1_sec4}
k_0+\log_2\log_{(1/\gamma)}\big({4\mu}/{L\epsilon}\big)=O\big(\log_2\log_2({1}/{\epsilon})\big).
\ene

It is worth noting that the above quadratic rate is \emph{global} and \emph{exact}, which is in sharp contrast with the existing distributed second-order methods \citep{qu2019finite,mokhtari2017network,mansoori2019fast,mansoori2017superlinearly,tutunov2019distributed,zargham2014accelerated}.

\begin{remark}[Communication complexity]
	Although each node in  DAN involves $n-1$ rounds of communicating messages of size $\cO(p^2)$ per iteration, the overall communication complexity could be much lower than that of first-order algorithms. In particular, the communication complexity in each node to achieve $\|\vx_k^{(i)}-\vx^\star\|\leq\epsilon$ is $O\big(np^2\log_2\log_2({1}/{\epsilon})\big)$ for DAN and  is $\cO(d_\cG p\log_2{1}/{\epsilon})$ for the optimal first-order algorithm \citep{scaman2017optimal}. The latter is higher than the former if a high precision is desired, i.e., $\epsilon\leq \cO(2^{-\frac{np}{d_\cG}})$.  \diff{A similar analysis can also be made for the computational complexity.}
\end{remark}



Despite the global quadratic convergence, DAN may progress relatively slow in the early stage due to the use of small stepsizes in the damped Newton phase ($k\leq k_0$) and the DSF in Line \ref{line3:dan} of Algorithm \ref{DAN}. \diff{In practice, we can start at a first-order method for a few iterations and then switch to DAN for quadratic convergence afterwards, which is similar to the practical use of the Newton method.} In Section \ref{sec5}, we  propose a novel communication-efficient DAN with low rank approximation.

\section{The Communication-efficient DAN-LA}\label{sec5}

In this section, we propose a novel communication-efficient version of DAN  in Algorithm \ref{DANLA} that reduces the transmitted Hessian to a vector of size $\cO(p)$. This reduction is significant in applications with high dimensional decision vectors. Since it maintains a global superlinear convergence rate, it is much faster than the existing distributed Newton-type algorithms \citep{mokhtari2017network,mansoori2019fast,mansoori2017superlinearly,tutunov2019distributed,mokhtari2016decentralized,eisen2018primal,varagnolo2016newton}.

To this end, we introduce two novel ideas  in Algorithm \ref{DANLA}: (a) a symmetric rank-$1$ matrix in $\bR^{p\times p}$ can be represented by the outer product of a vector in $\bR^{p}$ and itself with only a possible sign change; and (b) instead of directly compressing the Hessian, we approximate its {\em innovation} by a symmetric rank-1 matrix. If the decision vectors do not vary much between two consecutive iterations,  the innovation is expected to be ``small'', which suggests that  a rank-$1$ matrix approximation might not lose much information. 

Specifically, suppose that each node has the same estimate of the Hessian of $f$ in \eqref{obj} at the $(k-1)$-th iteration, say $\hat H_{k-1}^{(i)}$, and an estimate $H_{k-1}^{(i)}$ of the Hessian of $f_i$ such that $\hat H_{k-1}^{(i)}=\sum_{i=1}^{n}H_{k-1}^{(i)}$, which can be easily satisfied when $k=1$, e.g., let $\hat H_0^{(i)}=H_0^{(i)}=0,\forall i\in\cV$. At the next iteration, node $i$ computes a local Hessian $\nabla^2 f_i(\vx_{k}^{(i)})$ and approximates its innovation via a rank-1 matrix, i.e., 
$$
\nabla^2 f_i(\vx_{k}^{(i)})-H_{k-1}^{(i)} \approx s_{k}^{(i)}\cdot\vh_k^{(i)}(\vh_k^{(i)})^\T
$$
where $s_{k}^{(i)}$ and $\vh_k^{(i)}$ are solved by the Eckart-Young-Mirsky Theorem \citep[Theorem 2.23]{markovsky2012low}.  Particularly, let $\lambda_i$ be the $i$-th largest eigenvalue of $\nabla^2 f_i(\vx_{k}^{(i)})-H_{k-1}^{(i)}$ in magnitude and $\vw_i$ be the associated normalized right eigenvector, i.e., $(\nabla^2 f_i(\vx_{k}^{(i)})-H_{k-1}^{(i)} )\vw_i=\lambda_i \vw_i$ and $\|\vw_i\|=1$. Then, 
$\vh_k^{(i)}=\sqrt{|\lambda_{1}|}\vw_1,\ s_k^{(i)}=\text{sign}(\lambda_1),\ r_k^{(i)}=|\lambda_2|$ (Line \ref{line2:danla}). Clearly, this can be computed by the eigenvalue decomposition with complexity  $\cO(p^3)$ or a truncated singular value decomposition (SVD) with a complexity of $\cO(p^2)$ in some cases \citep{allen2016lazysvd}. Overall, the complexity is comparable to that of computing a Newton direction.

Moreover, node $i$ updates an estimate of the local Hessian as $H_{k}^{(i)}=H_{k-1}^{(i)}+s_{k}^{(i)}\cdot\vh_k^{(i)}(\vh_k^{(i)})^\T$(Line \ref{line3:danla}).  Via only communicating $s_{k}^{(i)}$ and $\vh_k^{(i)}$  in the DSF (Line \ref{line4:danla}), it is able to compute an estimate of  the Hessian of $f$  by $$\hat H_{k}^{(i)}=\hat H_{k-1}^{(i)}+\sum_{i=u}^{n}s_{k}^{(u)}\cdot\vh_k^{(u)}(\vh_k^{(u)})^\T \text{(Line \ref{line5:danla})}.  $$
The Newton direction is then computed based on the estimates $\hat H_{k}^{(i)}$ and $\hat\vg_k^{(i)}$.


Another challenge is how to handle the error in approximating the Hessian of $f$. Intuitively, if the error is large, a pure Newton may lead to divergence, and we adaptively turn to a higher rank approximation for a better estimate. Otherwise, we can safely run a Newton update. To achieve it, each node $i$ computes an upper bound on the estimation error $\hat r_k^{(i)}$ (Line \ref{line5:danla}). If the error is smaller than the explicit threshold $\underline{r}$ in \eqref{eq2_danla} (Line \ref{line6:danla}),  it updates via an adaptive Newton step \eqref{eq1_danla} with our novel stepsize (Line \ref{line7:danla}). \diff{Note that the invertibility of $\hat H_k^{(i)}$ in \eqref{eq1_danla} always holds, since $\|\hat H_k^{(i)}-\nabla^2 f(\vec{x}_k^{(i)})\|\leq \underline{r}\leq \mu/3$ and $\nabla^2 f(\vec{x}_k^{(i)})\succeq \mu I$ (c.f. Assumption \ref{assum1}(a)). }


%
%

\diff{More specifically, if $\alpha_k^{(i)}=0$ at the $k$-th iteration, then $\vx_{k+1}^{(i)}=\vx_{k}^{(i)}$ and we obtain from Lines \ref{line2:danla}-\ref{line3:danla} in Algorithm \ref{DANLA} that in the $(k+1)$-th iteration,  $\nabla^2 f_i(\vx_k^{(i)})-\left(H_{k-1}^{(i)}+s_k^{(i)}\vh_k^{(i)}(\vh_k^{(i)})^\T\right)\approx s_{k+1}^{(i)}\vh_{k+1}^{(i)}(\vh_{k+1}^{(i)})^\T$. Since $\vh_k^{(i)}$ and $\vh_{k+1}^{(i)}$ are two linearly independent vectors, it implies that $\nabla^2 f_i(\vx_k^{(i)})-H_{k-1}^{(i)}$ is in fact approximated by a rank-2 matrix $s_k^{(i)}\vh_k^{(i)}(\vh_k^{(i)})^\T+s_{k+1}^{(i)}\vh_{k+1}^{(i)}(\vh_{k+1}^{(i)})^\T$, and the approximation error is strictly reduced, i.e., $r_{k+1}^{(i)}<r_{k}^{(i)}$. Jointly with Lines \ref{line4:danla}-\ref{line5:danla}, there must exist a smallest nonnegative integer $m\le p-1$ such that $\hat{r}_{k+m}^{(i)}\le \underline{r}$ which implies that  $\alpha_{k+m}^{(i)}>0$, and eventually $\vx_{k}^{(i)}$ is updated to a new vector. One can easily see that the above is equivalent to the direct use of rank-$(m+1)$ approximation in \eqref{eq1_alg3}. }

\begin{algorithm*}[!t]
	\makeatletter
	\renewcommand\footnoterule{%
		\kern-3\p@
		\hrule\@width.4\columnwidth
		\kern2.6\p@}
	\makeatother
	\caption{The DAN-LA --- from the view of node $i$}\label{DANLA}
	\begin{minipage}{\linewidth}
		\begin{algorithmic}[1]
			\REQUIRE Starting point $\vx_{0}^{(i)}=\vx_0$, $\vg_0^{(i)}=\nabla f_i(\vx_{0}),\ \hat H_{-1}^{(i)}=H_{-1}^{(i)}=0$, $c>0$, and $\mu, L, M$ in Assumptions \ref{assum1} and \ref{assum2}.
			\FOR {$k=0,1,2,\cdots$}
			\STATE\label{line2:danla} Let $\lambda_j$ be the $j$-th largest eigenvalue in magnitude of $\nabla^2 f_i(\vx_k^{(i)})-H_{k-1}^{(i)}$, $\vw_j$ be the associated unit eigenvector and
		        \bee\label{eq1_alg3}
			\vh_k^{(i)}=\sqrt{|\lambda_{1}|}\vw_1,\ s_k^{(i)}=\text{sign}(\lambda_1),\ r_k^{(i)}=|\lambda_2|.
			\ene
			\vskip 0pt
			\STATE\label{line3:danla} Compute $H_{k}^{(i)}=H_{k-1}^{(i)}+s_k^{(i)}\vh_k^{(i)}(\vh_k^{(i)})^\T$ and $\vg_k^{(i)}=\nabla f_i(\vx_k^{(i)})$. \COMMENT{$H_k^{(i)}$ is the rank-1 approximation of $\nabla^2 f_i(\vx_k^{(i)})$}
			\vskip 2pt
			\STATE\label{line4:danla} Run Algorithm \ref{DSF} to obtain $\cS=\text{DSF}(r_k^{(i)},s_k^{(i)},\vg_k^{(i)},\vh_k^{(i)})$,	where $\cS=\{(r_k^{(u)},s_k^{(u)},\vg_k^{(u)},\vh_k^{(u)})|u\in\cV\}$.\\ \COMMENT{Finite-time set-consensus}
			\vskip 2pt
			\STATE\label{line5:danla} Use $\cS$ to compute $\hat\vg_k^{(i)}=\sum_{u=1}^{n}\vg_k^{(u)}$, $\hat H_{k}^{(i)}=\hat H_{k-1}^{(i)}+\sum_{u=1}^{n}s_k^{(u)}\vh_k^{(u)}(\vh_k^{(u)})^\T$, and $\hat r_k^{(i)}=\sum_{u=1}^{n}r_k^{(u)}$.\\ \COMMENT{$\hat H_k^{(i)}$ is an approximation of $\nabla^2 f(\vx_k^{(i)})$ and $\hat r_i^{(i)}$ bounds the approximation error}
			\vskip 2pt
			\STATE\label{line6:danla} Let $M_c=M+c$. Set the stepsize \COMMENT{Update only if the approximation error is small}
			\bee\label{eq2_danla}
			\alpha_{k}^{(i)}=\begin{cases}
				\min\Big\{1,\frac{\phi}{\|\hat\vg_{k}^{(i)}\|}\Big\}\text{, where }\phi\triangleq\frac{2\mu(\mu-\underline{r})^2}{L(M+\mu)}-\frac{2\underline{r}(\mu-\underline{r})}{L}>0, & \text{if $\hat r_k^{(i)}\leq\underline{r}\triangleq\frac{1}{3}\sqrt{M_c^2+3\mu^2}-\frac{M_c}{3}$} \\
				0,                                                                                                                                                                           & \text{otherwise.}
			\end{cases}
			\ene
			\STATE\label{line7:danla} \diff{If $\alpha_{k}^{(i)}=0$, set $\vx_{k+1}^{(i)}=\vx_{k}^{(i)}$. Otherwise}, update \COMMENT{Adaptive Newton step}
			\bee\label{eq1_danla}
			\vx_{k+1}^{(i)}=\vx_{k}^{(i)}-\alpha_k^{(i)} (\hat H_{k}^{(i)})^{-1}\hat\vg_{k}^{(i)}.
			\ene
			\ENDFOR
		\end{algorithmic}
	\end{minipage}
\end{algorithm*}

We prove the convergence of DAN-LA under the standard assumption in quasi-Newton methods, see e.g. \cite{nocedal2006numerical}.
\begin{assum}\label{assum2}
	The Hessian of $f$ in \eqref{obj} is upper bounded by $M$, i.e., $\nabla^2 f(\vx)\preceq MI,\forall \vx\in \mathbb{R}^p$.
\end{assum}



\begin{theo}\label{theo3}
	Suppose Assumptions \ref{assum1}-\ref{assum2} hold, and let $\{\vx_k^{(i)}\}$ be generated by DAN-LA. For any $c> 0$ in Line \ref{line6:danla} of Algorithm \ref{DANLA} and any $i\in\cV$,  $\|\vx_k^{(i)}-\vx^\star\|$ converges superlinearly to 0, where $\vx^\star$ is the optimal point of \eqref{obj}.
\end{theo}

Similar to quasi-Newton methods \citep{nocedal2006numerical}, it is very difficult to explicitly quantify its convergence rate. Some remarks are provided after the proof.

\begin{proof}[Proof of Theorem \ref{theo3}]
	We first prove the convergence of $\{\vx_k^{(i)}\}$ and $\{r_k^{(i)}\}$, and then invoke Dennis-Mor\'{e} Theorem \citep{dennis1974characterization} to prove global superlinear convergence.

	{\bf Step 1:  A key inequality.}
	It can readily be obtained by mathematical induction that $\vx_k^{(1)}=\cdots=\vx_k^{(n)}\triangleq\vx_k$, $\hat\vg_k^{(1)}=\cdots=\hat\vg_k^{(n)}=\vg_k=\nabla f(\vx_k)$, $\alpha_k^{(1)}=\cdots=\alpha_k^{(n)}\triangleq\alpha_k$ and $\hat H_k^{(1)}=\cdots=\hat H_k^{(n)}\triangleq\hat H_k$. That is, all agents actually perform identical updates with aid of the set-consensus algorithm, and  \eqref{eq1_danla} becomes
	\bee\label{eq1_pfdanla}
	\vx_{k+1}=\vx_{k}-\alpha_k (\hat H_{k})^{-1}\vg_{k}.
	\ene
	Moreover, it implies that  $\hat H_k = \sum_{i=1}^{n}H_k^{(i)}$ (c.f. Lines \ref{line3:danla} and \ref{line5:danla}). By the Eckart-Young-Mirsky Theorem \citep[Theorem 2.23]{markovsky2012low}, we have $\|\nabla^2 f_i(\vx_k)-H_k^{(i)}\|= r_k^{(i)}$ and
	  	\bea\label{eq2_pfdanla}
	&\|\nabla^2 f(\vx_k)-\hat H_k\|=\Big\|\sum_{i=1}^{n}\nabla^2 f_i(\vx_k)-H_k^{(i)}\Big\|\\
	&\leq\sum_{i=1}^{n}\|\nabla^2 f_i(\vx_k)-H_k^{(i)}\|=\hat r_k^{(i)}.
	\ena
	Recall that $\vg_k=\nabla f(\vx_k)$, and let
	\bee\label{eq4_pfdanla}
	\vz_k=\vx_{k+1}-\vx_k=-\alpha_k (\hat H_{k})^{-1}\vg_{k}.
	\ene
	Then, we have
	\bea\label{eq1_proof}
	\vg_{k+1}&=\vg_k+\int_{0}^{1}\nabla^2 f(\vx_k+t\vz_k)\vz_kdt\\
	&=\vg_k+\int_{0}^{1}\big(\nabla^2 f(\vx_k+t\vz_k)-\hat H_k\big)\vz_kdt+\hat H_k\vz_k\\
	&=(1-\alpha_k)\vg_k+\int_{0}^{1}\Big(\nabla^2 f(\vx_k)-\hat H_k\Big)\vz_kdt\\
	&\quad+\int_{0}^{1}\Big(\nabla^2 f(\vx_k+t\vz_k)-\nabla^2 f(\vx_k)\Big)\vz_kdt
	\ena
	We now show that $\alpha_k\in(0,1]$ by proving that $\phi$ in \eqref{eq2_danla} is positive. Let $\tilde\phi(\underline{r})\triangleq\frac{L}{2(\mu-\underline{r})^2}\phi=\frac{\mu}{M+\mu}-\frac{\underline{r}}{\mu-\underline{r}}$. Then, $\tilde\phi(\underline{r})$ is strictly decreasing on $(0,\mu)$.  Since $\underline{r}=\frac{1}{3}\sqrt{M_c^2+3\mu^2}-\frac{M_c}{3}={\mu^2}/{(\sqrt{M_c^2+3\mu^2}+M_c)}<{\mu^2}/{(M+2\mu)}<\mu$, then $\tilde\phi(\underline{r})>\tilde\phi(\frac{\mu^2}{M+2\mu})=0$, i.e., $\phi>0$.

	It follows from \eqref{eq1_proof} that
	\bea\label{eq2_proof}
	\|\vg_{k+1}\|&\leq (1-\alpha_k)\|\vg_k\|+\|\nabla^2 f(\vx_k)-\hat H_k\|\|\vz_k\|\\
	&\quad+\int_{0}^{1}\Big\|\nabla^2 f(\vx_k+t\vz_k)-\nabla^2 f(\vx_k)\Big\|\|\vz_k\|dt\\
	&\leq (1-\alpha_k)\|\vg_k\|+\hat r_k^{(i)}\|\vz_k\|+\frac{L}{2}\|\vz_k\|^2
	\ena
	where we used \eqref{eq2_pfdanla} and Assumption \ref{assum1} in the last inequality.

	{\bf Step 2:  prove that $\lim_{k\rightarrow\infty}\{\alpha_k\}=1$ and $\{\vx_k\}$ converges.}
	Let $\cK=\{k_l:l\geq 0\}$ be an increasing sequence such that $k\in\cK$ if and only if $\alpha_{k}>0$ at iteration $k$, which follows from \eqref{eq2_danla} that $\hat r_{k}^{(i)}\leq\underline{r}$. 
	
	We show that $k_{l+1}-k_l\leq p$ and $\vx_{k_l+1}=\vx_{k_l+2}=\cdots=\vx_{k_{l+1}}$, where $p$ is the dimension of $\vx_k$.  For any $k_l<k< k_{l+1}$, it follows from \eqref{eq1_pfdanla} and $\alpha_k=0$ that $\vx_{k+1}=\vx_{k}$, and the second part follows by induction.  Suppose that there exists an $l$ such that $k_{l+1}-k_l>p$.  Let $\lambda_k^{(j)},j=1,\cdots,p$ be the $j$-th largest eigenvalue of $\nabla^2 f_i(\vx_k^{(i)})-H_{k-1}^{(i)}$ in magnitude and $\vw_k^{(j)}$ be the corresponding unit eigenvector. Define $\lambda_{k}^{(j)}=0$ for any $j>p$. Since $\nabla^2 f_i(\vx_k^{(i)})=\nabla^2 f_i(\vx_{k+1})$ for any $k_l<k<k_{l+1}$, it follows from \eqref{eq1_alg3} that
	\bea
	&\nabla^2 f_i(\vx_{k+1})-H_k^{(i)}\\
	&=\nabla^2 f_i(\vx_{k})-H_{k-1}^{(i)}-\lambda_{k}^{(1)}\vw_k^{(1)}(\vw_k^{(1)})^\T\\
	&=\sum_{j=2}^{p}\lambda_{k}^{(j)}\vw_k^{(j)}(\vw_k^{(j)})^\T,
	\ena
	and hence $\lambda_{k+1}^{(j)}=\lambda_{k}^{(j+1)}$ for all $j>0$, which further implies that $|\lambda_{k+1}^{(j)}|\leq|\lambda_{k}^{(j)}|$.  Then, for any $k\in[k_l+p,k_{l+1})\neq\emptyset$, we have $r_k^{(i)}=|\lambda_{k}^{(2)}|\leq|\lambda_{k_l+p}^{(2)}|=|\lambda_{k_l+1}^{(p+1)}|=0,\forall i$. Thus, $\hat r_k^{(i)}=\sum_{i=1}^{n}r_k^{(i)}=0$ and we must have $\alpha_k>0$, and hence $k\in\cK$, which leads to a contradiction. 

	Next, we study the sequence $\{\|\vg_{k_l}\|:k_l\in\cK\}$. Consider the first stage that  $\alpha_{k_l}={\phi}/{\|\hat\vg_{k_l}^{(i)}\|}\leq1$. It follows from \eqref{eq2_pfdanla} and the definition of $\cK$ that
	\bee\label{eq3_pfdanla}
	\|\nabla^2 f(\vx_{k_l})-\hat H_{k_l}\|\leq\hat r_{k_l}^{(i)}\leq\underline{r}=\frac{\underline{r}\alpha_{k_l}\|\hat\vg_{k_l}^{(i)}\|}{\phi}
	\ene
	which jointly with Assumption \ref{assum2} implies that $\|\hat{H}_{k_l}\|\leq \|\nabla^2 f(\vx_{k_l})\|+\underline{r}\leq M+\underline{r}$. Thus, it follows from \eqref{eq4_pfdanla} that
 $
	\|\vz_{k_l}\|=\alpha_{k_l}\|(\hat{H}_{k_l})^{-1}\vg_{k_l}\|=\alpha_{k_l}\sqrt{(\vg_{k_l})^\T(\hat{H}_{k_l})^{-2}\vg_{k_l}}
	\geq {\alpha_{k_l}\|\vg_{k_l}\|}/{(M+\underline{r})}
$
	which, together with \eqref{eq3_pfdanla}, implies that
	\bee\label{eq5_pfdanla}
	\hat r_{k_l}^{(i)}\leq\frac{\underline{r}\alpha_{k_l}\|\hat\vg_{k_l}^{(i)}\|}{\phi}\leq\frac{\underline{r}(M+\underline{r})\|\vz_{k_l}\|}{\phi}.
	\ene
	Combining \eqref{eq5_pfdanla} and \eqref{eq2_proof} yields that
	\bea\label{eq7_pfdanla}
	&\|\vg_{k_{l+1}}\|=\|\vg_{{k_l}+1}\|\\
	&\leq(1-\alpha_{k_l})\|\vg_{k_l}\|+\Big(\frac{L}{2}+\frac{\underline{r}(M+\underline{r})}{\phi}\Big)\|\vz_{k_l}\|^2\\
	&\leq\|\vg_{k_l}\|-\alpha_{k_l}\|\vg_{k_l}\|+\frac{L\phi+2\underline{r}(M+\underline{r})}{2\phi(\mu-\underline{r})^2}(\alpha_{k_l}\|\vg_{k_l}\|)^2\\
	&=\|\vg_{k_l}\|-\phi+\frac{L\phi+2\underline{r}(M+\underline{r})}{2(\mu-\underline{r})^2}\phi
	\ena
	where the first equality follows from $\vx_{k_l+1}=\vx_{k_l+2}=\cdots=\vx_{k_{l+1}}$, the second inequality follows from \eqref{eq4_pfdanla} and \eqref{eq3_pfdanla}, and the last equality follows from the definition of the first stage.

Let $\theta={\mu}/{(M+\mu)}$. The definition of $\phi$ implies that
	\bee\label{eq10_pfdanla}
	\frac{L\phi}{2(\mu-\underline{r})^2}=\theta-\frac{\underline{r}}{\mu-\underline{r}}.
	\ene
	Let $	q(r)\triangleq-\frac{r}{\mu-r}+\frac{r(M+r)}{(\mu-r)^2}=\frac{2r^2+(M-\mu)r}{(\mu-r)^2}$ which is strictly increasing on $[0,\mu)$.  Notice that $q(0)=0$ and $q(r_0)={1}/{2}$ with $r_0={1}/{3}\cdot(\sqrt{M^2+3\mu^2}-{M})$. Since $0<\underline{r}<r_0$, we have $q(\underline{r})<q(r_0)={1}/{2}$. Since $M\geq\mu$, it jointly with  \eqref{eq10_pfdanla} implies that
	\bee\label{eq6_pfdanla}
	\frac{L\phi+2\underline{r}(M+\underline{r})}{2(\mu-\underline{r})^2}<\frac{\mu}{M+\mu}+\frac{1}{2}\leq1.
	\ene
	
By \eqref{eq7_pfdanla}, we obtain that
	\bee\label{eq9_pfdanla}
	\|\vg_{k_{l+1}}\|\leq\|\vg_{k_l}\|-\Big(1-\frac{L\phi+2\underline{r}(M+\underline{r})}{2(\mu-\underline{r})^2}\Big)\phi.
	\ene
	Thus, the sequence $\{\|\vg_{k_l}\|\}$ is monotonically decreasing by at least a  constant at each step in the first stage. \eqref{eq9_pfdanla} implies that there exists a $k_l$ such that the algorithm enters the second stage where $\alpha_{k_l}=1$, i.e., $\|\vg_{k_l}\|\leq\phi$. Similar to \eqref{eq7_pfdanla}, it follows from \eqref{eq2_proof} that
	\bea\label{eq11_pfdanla}
	\|\vg_{k_{l+1}}\|&=\|\vg_{{k_l}+1}\|\leq\hat r_k^{(i)}\|\vz_{k_l}\|+\frac{L}{2}\|\vz_{k_l}\|^2\\
	&\leq\frac{\underline{r}}{\mu-\underline{r}}\|\vg_{k_l}\|+\frac{L}{2(\mu-\underline{r})^2}\|\vg_{k_l}\|^2.
	\ena
	Let $e_{k_l}=\|\vg_{k_l}\|/\phi\leq 1$. Then, it implies that
	\bea\label{eq8_pfdanla}
	e_{k_{l+1}}&\leq\frac{\underline{r}}{\mu-\underline{r}}e_{k_l}+\frac{L\phi}{2(\mu-\underline{r})^2}(e_{k_l})^2\\
	&\leq\Big(\frac{\underline{r}}{\mu-\underline{r}}+\frac{L\phi}{2(\mu-\underline{r})^2}\Big)e_{k_l}.
	\ena
	Recall from \eqref{eq10_pfdanla} that
	\bee\label{eq12_pfdanla}
	\frac{\underline{r}}{\mu-\underline{r}}+\frac{L\phi}{2(\mu-\underline{r})^2}=\theta\leq\frac{1}{2},
	\ene
	which, combined with \eqref{eq8_pfdanla}, shows that $\{e_{k_l}\}$ and $\{\|\vg_{k_l}\|\}$ are monotonically decreasing and converge to 0 at least linearly at the second stage. This also concludes that $\|\vg_{k}\|$ will remain in the second stage for all $k\geq k_l$, and hence $\lim_{k\ra\infty}\alpha^k=1$.

	Combining all above, it implies that $\{\|\vg_{k_l}\|:k_l\in\cK\}$ converges to 0. Since $k_{l+1}-k_l\leq p$ and $\vx_{k_l+1}=\cdots=\vx_{k_{l+1}}$, then $\|\vg_{k}\|$ converges to 0 as well.  Since $\|\vx_k-\vx^\star\|\leq\frac{1}{\mu}\|\vg_k\|$, we have thus proved that $\vx_{k}$ converges to $\vx^\star$.

	{\bf Step 3: prove the convergence of $\{r_k^{(i)}\}$ to 0.} We now show that the approximation error $\{r_k^{(i)}\}$ converges to 0 for all $i$. Denote by $\sigma_j(A)$ the $j$-th largest singular value of $A\in\bR^{p\times p}$. Let $E_k=\nabla^2 f_i(\vx_k^{(i)})-H_{k-1}^{(i)}$ and $\varepsilon_k=\|E_k\|_\ast$, where $\|A\|_\ast=\sum_{j=1}^{p}\sigma_j(A)\geq0$ is the nuclear norm of $A$. Since the algorithm will enter the second stage after a finite number of iterations, it is sufficient to focus only on this stage. Denote by $k_l\in\cK$ the first time step when the algorithm enters it. Then,
	\bea\label{eq13_pfdanla}
	&\varepsilon_{{k_{l+1}}} \\
	&\leq\varepsilon_{{k_l}+1}=\|\nabla^2 f_i(\vx_{k_l+1}^{(i)})-H_{k_l}^{(i)}\|_\ast\\
	&=\|\nabla^2 f_i(\vx_{k_l+1}^{(i)})-\nabla^2 f_i(\vx_{k_l}^{(i)})+\nabla^2 f_i(\vx_{k_l}^{(i)})-H_{k_l}^{(i)}\|_\ast\\
	&\leq\|\nabla^2 f_i(\vx_{k_l}^{(i)})-H_{k_l}^{(i)}\|_\ast+\|\nabla^2 f_i(\vx_{k_l+1}^{(i)})-\nabla^2 f_i(\vx_{k_l}^{(i)})\|_\ast\\
	&\leq\sum_{j=2}^{p}\sigma_j(E_{k_l})+p\|\nabla^2 f_i(\vx_{k_l+1}^{(i)})-\nabla^2 f_i(\vx_{k_l}^{(i)})\|\\
	&\leq\varepsilon_{{k_l}}-\sigma_1(E_{k_l})+\frac{pL}{\mu-\underline{r}}\|\vg_{k_l}\|\\
	&\leq (1-\frac{1}{p})\varepsilon_{{k_l}}+\frac{pL}{\mu-\underline{r}}\|\vg_{k_l}\|
	\ena
	where we used $\|A\|_\ast\leq p\|A\|$, Assumption \ref{assum1}, the rank-1 approximation property, and $0\leq \varepsilon_k\leq p\sigma_1(E_k)$.	Note that \eqref{eq11_pfdanla}-\eqref{eq12_pfdanla} show that $\|\vg_{k_l}\|,k_l\in\cK$ converges to 0 at least linearly. 
	Then, it follows from \eqref{eq13_pfdanla} that $\lim_{k\ra\infty}\varepsilon_k=0$. The convergence of $\{r_k^{(i)}\}$ to 0 follows immediately by noticing $\varepsilon_k\geq r_k^{(i)}$.


	{\bf Step 4: prove the superlinear convergence.} Now we study the sequence $\hat H_k$. It follows from \eqref{eq2_pfdanla} and Assumption \ref{assum1} that
	\bea
	\|\hat H_k-\nabla^2 f(\vx^\star)\|&\leq\|\hat H_k-\nabla^2 f(\vx_k)\|+L\|\vx_k-\vx^\star\|\\
	&\leq\sum_{i=1}^{n}r_k^{(i)}+L\|\vx_k-\vx^\star\|.
	\ena
	Since we have already shown that both $\|\vx_k-\vx^\star\|$ and $\sum_{i=1}^{n}r_k^{(i)}$ converge to 0, then $\hat H_k$ converges to $\nabla^2 f(\vx^\star)$. In view of \eqref{eq1_pfdanla}, it  follows from Lemma \ref{lemma_dennis} that the superlinear convergences of $\|\vx_k-\vx^\star\|$ and $\|\nabla f(\vx_k)\|$.  \end{proof}

\begin{remark}[Relation with BFGS]\label{remark_quasi}
	\diff{Even though DAN-LA is also a  quasi-Newton method, it does not imply that we can combine DSF and other quasi-Newton methods (e.g. BFGS \citep{nesterov2018lectures,nocedal2006numerical}) to achieve the global superlinear convergence. As it usually requires using a line-search scheme, one of the main challenges lies in the design of stepsizes in the distributed setting.  A striking difference from BFGS is that the Hessian approximation error in DAN-LA can be explicitly controlled, which is essential to our stepsize co-design (Line \ref{line6:danla}).  
	
	In fact, BFGS has been originally adopted to avoid computing a Hessian, whereas the Hessian approximation in DAN-LA aims to reduce the communication cost and nodes still need to compute the Hessian.  }

\end{remark}

\begin{remark}[The effect of parameter $c$]\label{remark5}
	The parameter $c$ in Line \ref{line6:danla} aims to balance the communication cost and the computational cost. Specifically, from \eqref{eq2_danla}, a larger $c$ means a better approximation of the  Hessian of $f$ and also suggests a larger stepsize. However, it may result in a larger number of un-updated iterations. How the parameter $c$ affects the global convergence rate is related to the network bandwidth, parameters in Assumption \ref{assum1}, etc. In fact, DAN-LA reduces to DAN if $c$ tends to infinity, and it is also feasible to set $c=0$ if $M$ is \emph{strictly} larger than $\mu$. We will empirically show its effect in Section \ref{sec6}.
\end{remark}

\begin{remark}[Comparison with DAN]
	Compared to DAN, DAN-LA reduces the transmitted messages' size from $\cO(p^2+p)$ to $\cO(2p+1)$ per iteration, which is essentially identical to existing first-order methods, e.g., \citet{nedic2017achieving} requires nodes to transmit messages with size $\cO(2p)$. Nevertheless, DAN-LA may require more iterations than DAN to achieve the same level of accuracy. It is difficult to conclude which one is always better from a theoretical point of view. In our experiments, we find that the number of total transmitted messages in bits of DAN-LA is much smaller than that of DAN to achieve the same level of accuracy, while the computation increases.
\end{remark}

\begin{remark}[Importance of set-consensus]
	\diff{Although the finite-time set-consensus step in DAN can be replaced by a finite-time \emph{average} consensus \citep[e.g.][]{charalambous2018laplacian,wang2018speeding,yuan2013decentralised} to reduce memory size, it is indispensable in DAN-LA.} In Line \ref{line5:danla} of Algorithm \ref{DANLA}, nodes need to compute the summation $\frac{1}{n}\sum_{i=1}^{n}s_k^{(i)}\vh_k^{(i)}(\vh_k^{(i)})^\T$, which requires transmitting messages in size $\cO(p^2)$ to their neighbors if iterative consensus methods are adopted. In contrast, the transmitted messages are of size $\cO(p)$ in the DSF.
\end{remark}

\begin{remark} \label{remark9}
	Similar to DAN, it is suggested to initially perform several iterations of first-order methods to achieve a good starting point for DAN-LA. In view of Line \ref{line5:danla}, the complexity of obtaining the inverse  in \eqref{eq1_danla} can be reduced from $\cO(p^3)$ to $\cO(np^2)$ if $n<p$ by invoking the Sherman-Morrison-Woodbury formula \cite[Section 0.7.4]{horn2012matrix}. 
\end{remark}

\section{Numerical Examples}\label{sec6}

In this section, we test DAN and DAN-LA by training a binary logistic regression classifier for the Covertype datatset from the UCI machine learning repository \citep{Dua2017UCI}, where the samples in Classes 3 and 7 are used. The optimization problem involved has the following form: 
\bee
\text{min.}\ l(\omega)\triangleq-\sum_{i=1}^{m}y_i\ln\sigma(z_i)+(1-y_i)\ln(1-\sigma(z_i))+\frac{\rho}{2}\|\omega\|^2
\ene
where $\omega\in\bR^{55}$ and $m=56264$ is the number of samples; $z_i=\omega^\T\vx_i$ where $\vx_i\in\bR^{55}$ is the feature of the $i$-th sample with each entry normalized to $[-1,1]$, and $y_i\in\{0,1\}$ is the corresponding label. The regularization parameter is chosen as $\rho=0.01m$. The gradient and Hessian are respectively $\nabla l(\omega)=\sum_{i=1}^m\vx_i(\sigma(z_i)-y_i)+r\omega$ and  $\nabla^2 l(\omega)=\sum_{i=1}^m\vx_i\vx_i^\T\sigma(z_i)(1-\sigma(z_i))+rI$.

For distributed training, we randomly partition the dataset over $n=10$ or $n=100$ nodes with each one privately holding a local subset. We compare our algorithms with the four second-order methods: FNRC \citep{varagnolo2016newton}, ESOM-3 \citep{mokhtari2016decentralized}, \diff{Newton tracking \citep{zhang2020newton},  D-BFGS \citep{eisen2017decentralized},} and a first-order method: DIGing \citep{nedic2017achieving}. An undirected communication network is constructed by  adopting the Erd\H os-R\' enyi model \citep{erdHos1960evolution}, i.e.,  each pair of nodes is connected with probability $2\ln n/n$. For comparison, the edge weights are generated by the Metropolis method \citep{nedic2017achieving,shi2015extra}. 

We also implement the centralized gradient descent method as a baseline, where the training is conducted on a single node. In all algorithms, we set $\mu=0.02m, L=m$ and $M=0.04m$ in Assumptions \ref{assum1} and \ref{assum2} to guide the selection of stepsizes. For example, the stepsize in the centralized gradient method is set to the optimal one $2/(\mu+L)$ \citep{nesterov2018lectures} and the stepsize used in DIGing is $a/L$ with $a$ determined by a grid search. If the theoretical suggestions are not clear, we manually tune the stepsizes to obtain the numerically best one.
The convergence behaviors of these algorithms are depicted in Figs. \ref{fig1} and \ref{fig2}. 

\begin{figure}[!t]
	\centering
	\subfloat[$n=10$]{\label{fig1a}\includegraphics[width=0.96\linewidth]{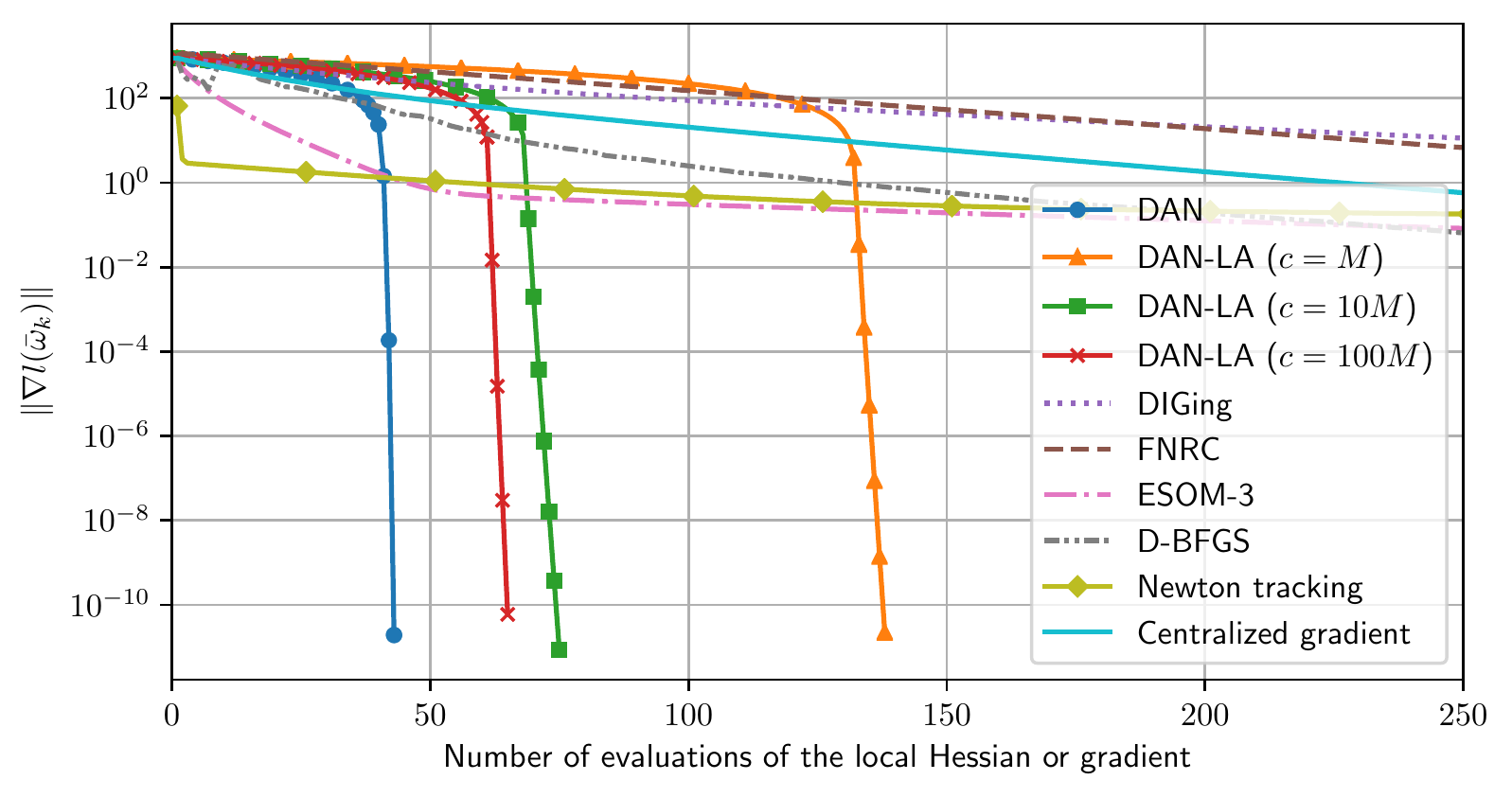}}\\
	\subfloat[$n=100$]{\label{fig1b}\includegraphics[width=0.96\linewidth]{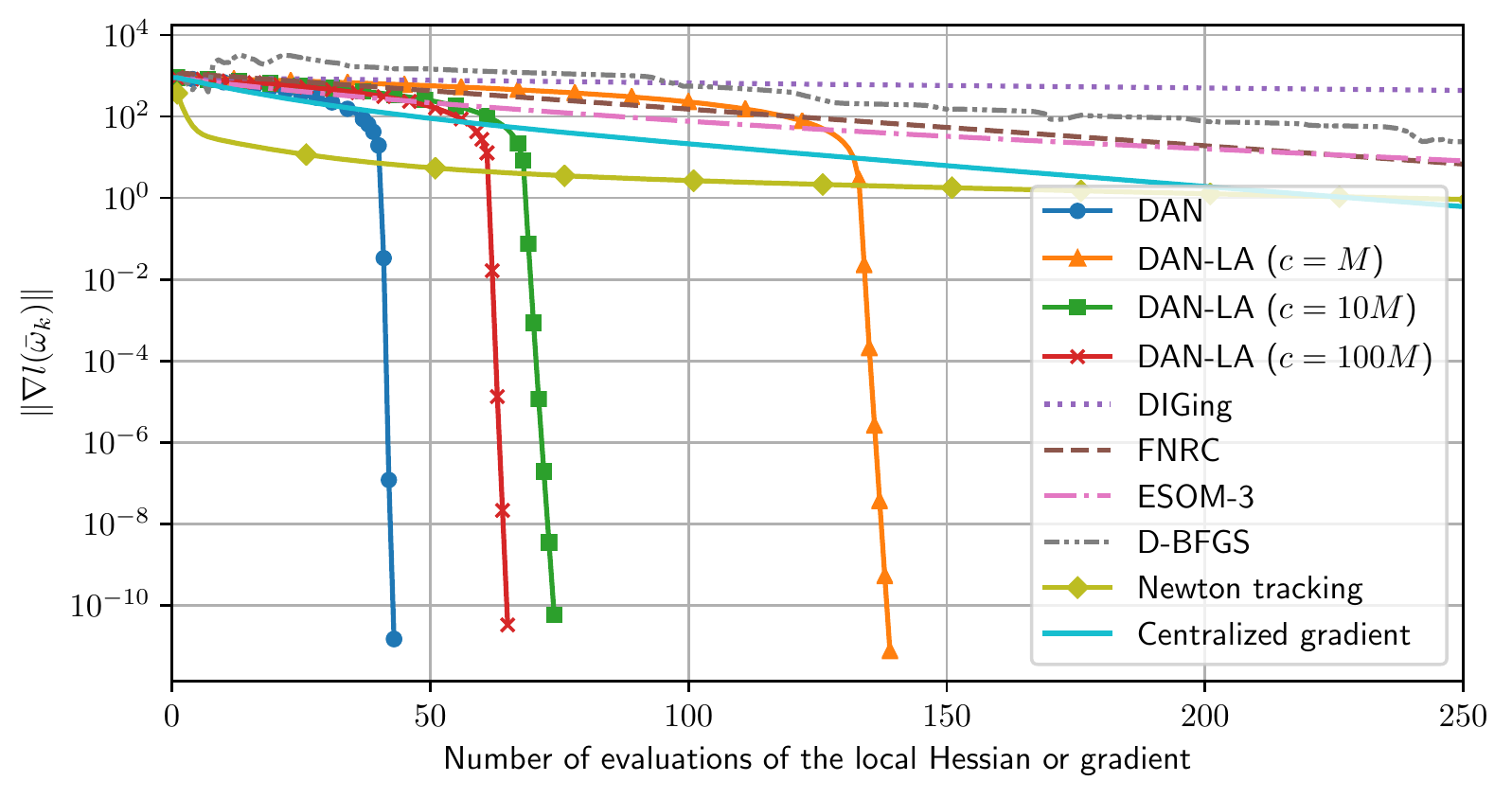}}
	\caption{The convergence rates w.r.t. the number of evaluations of local Hessian or gradient.}
	\label{fig1}
\end{figure}

\begin{figure}[!t]
	\centering
	\subfloat[$n=10$]{\label{fig2a}\includegraphics[width=0.96\linewidth]{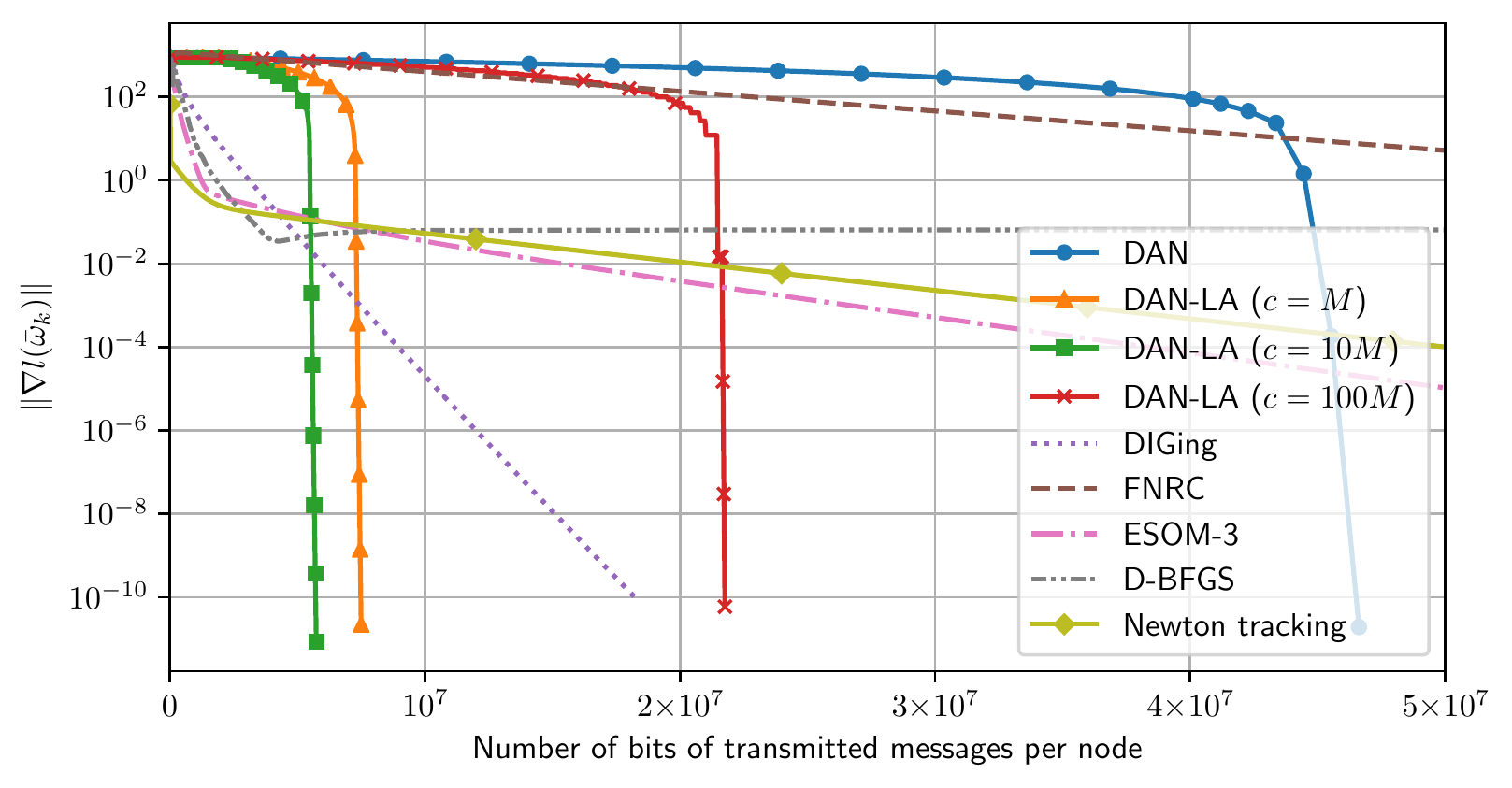}}\\
	\subfloat[$n=100$]{\label{fig2b}\includegraphics[width=0.96\linewidth]{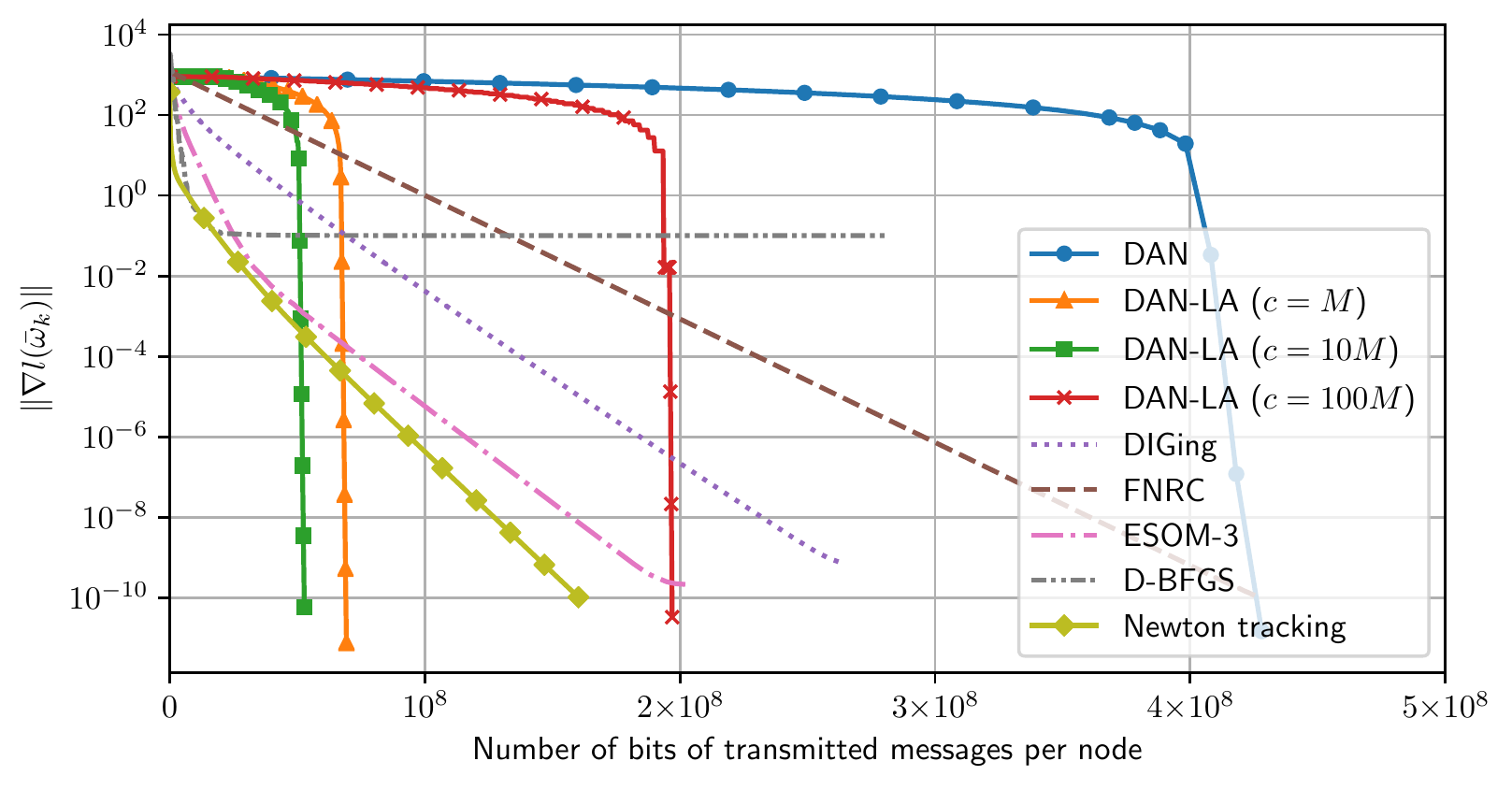}}
	\caption{The convergence rate w.r.t. the transmitted bits.}
	\label{fig2}
\end{figure}

Fig. \ref{fig1} plots the convergence rates for the cases of $n=10$ and $n=100$, respectively. The rate is measured by the decreasing speed of the global gradient's norm versus the number of evaluations of the local gradient or Hessian. It shows that both DAN and DAN-LA achieve superlinear convergence and outperform DIGing, FNRC, D-BFGS, Newton tracking, and ESOM-3. Though they progress slow  in the early stages, they rapidly outpace others in the pure Newton phase as expected. In practice, one can adopt some first-order methods initially and then switch to DAN or DAN-LA (c.f. Remark \ref{remark9}). Fig. \ref{fig1} validates Remark \ref{remark5} where a larger $c$ leads to a better approximation quality of the global Hessian and a larger stepsize. 

Fig. \ref{fig2} shows the convergence rates versus the averaged transmitted bits of a node (a number is stored and transmitted in 64-bit floating-point format). The DAN and DAN-LA behave similarly as in Fig. \ref{fig1}. Moreover, the transmission of Hessians results in large communication overhead of DAN before it enters the pure Newton phase. Thus, DAN seems more suitable for networks with high bandwidth in modest size, and DAN-LA with an appropriate $c$ is much more communication efficient.  
\section{Conclusion}\label{sec7}

This paper has proposed two distributed second-order optimization algorithms with global superlinear convergence. The striking features lie in the use of (a) a finite-time set-consensus method, (b)  an adaptive version of Newton method for global convergence, and (c) the low-rank matrix approximation methods to compress the Hessian for efficient communication.  Future works can focus on asynchronous versions of the proposed algorithms as in \citep{zhang2019asynchronous,zhang2019asyspa}, the integration of the proposed algorithms with quasi-Newton methods, and developing more communication-efficient ones.


\appendix\label{apx}
\section{Technical Lemma}

\begin{lemma}[\citealp{dennis1974characterization}]\label{lemma_dennis}
		Suppose $f$ is twice continuously differentiable and $\nabla^2 f(\vx^\star)$ is nonsingular for some $\vx^\star$. Let $\{B_k\}$ be a sequence of nonsingular matrices, and for some $\vx_0$ let the sequence  $\{\vx_k\}$ converge to $\vx^\star$, where $\vx_{k+1}=\vx_k-\alpha_k(B_k)^{-1}\nabla f(\vx_k)$. Then, $\{\vx_k\}$ converges superlinearly to $\vx^\star$ and $\nabla f(\vx^\star)=0$ if and only if $\lim_{k\ra\infty}\alpha_k=1$ and
		\bee
		\lim_{k\ra\infty}\frac{\|[B_k-\nabla^2 f(\vx^\star)](\vx_{k+1}-\vx_{k})\|}{\|\vx_{k+1}-\vx_{k}\|}=0.
		\ene
	\end{lemma}
\section{Proof of Theorem \ref{theo1}}\label{apx}
	The proof for directed networks is given in \citet[Theorem 1]{oliva2017distributed}, and here we just deal with undirected trees. Consider a tree with the root an arbitrary node. Let $\cU_v(k)=\{u|\cS_u\in\cI_v(k)\}$ be the set of nodes whose messages arrive at node $v$ in $k$ iterations, $\cR_v$ be the set of \emph{descendants} of $v$ (not including $v$), and $\widetilde\cU_v(k)=\cU_v(k)\cap\cR_v$ be the set of descendants of $v$ whose messages are in $\cU_v(k)$. It follows from the definition that $\widetilde\cU_v(k)\subseteq \cR_v$ and $\widetilde\cU_v(k)\subseteq\cU_v(k)$, and $\cR_v=\emptyset$ for a leaf node $v$.

	\emph{Claim 1:} If  $v$ is not a leaf and $u$ is a child of $v$, then $|\cU_v(k)\cap\cR_u|\leq k-1$.

	\emph{Proof:} Since $\cG$ is a tree, node $v$ can get the messages of nodes in $\cR_u$ only from node $u$. 
	Note that node $u$ sends the message of itself to node $v$ at the first iteration. The result then follows from that only $k-1$ iterations can be used to transmit the messages of nodes in $\cR_u$, and each iteration only transmits one message. $\hfill\square$ 

	\emph{Claim 2:} $|\widetilde\cU_v(k)|\geq \min\{|\cR_v|,k\}$ for all $v\in\cV$ and $k\in\bN$.

	\emph{Proof:} We prove it by mathematical induction. It is clear that $|\widetilde\cU_v(0)|=0$ for all $v\in\cV$. Suppose that $|\widetilde\cU_v(t)|\geq\min\{|\cR_v|,t\}$ for all $v\in\cV$, and we next show that $|\widetilde\cU_v(t+1)|\geq\min\{|\cR_v|,t+1\}$. For some node $v$, if $|\widetilde\cU_v(t)|> t$ or $|\widetilde\cU_v(t)|=|\cR_v|$, then the result follows immediately. Now consider  $|\widetilde\cU_v(t)|= t<|\cR_v|$, which means that $v$ cannot be a leaf. It follows that there exists a child $u$ of node $v$ and a nonempty subset $\cR'\subset\cR_u$ such that all elements in  $\cR'$ are not contained in $\cU_v(k)$, i.e., $\cU_v(k)\cap\cR'=\emptyset$ (otherwise $|\widetilde\cU_v(t)|=|\cR_v|$). Then, the hypothesis implies that $|\cR_u|\geq |\widetilde\cU_u(t)|\geq \min\{|\cR_u|,t\}$. If $|\widetilde\cU_u(t)|=|\cR_u|$, then $\cR'\subset\cU_u(t)$. Note that the selective operation ensures that no messages are transmitted more than once over a link $(u,v)$, which implies that an element in $\cR'$ will be sent to node $v$ at iteration $t+1$, and hence $|\cU_v(t+1)|\geq t+1$. If $|\widetilde\cU_u(t)|\geq t$, which means $|\cR_u|\geq t$, then it follows from Claim 1 that there must exist a node in $\widetilde\cU_u(k)$  whose message is not contained in $\cU_v(k)$, and hence this element will be sent to node $v$ from node $u$ at iteration $t+1$ due to the selective operation. Therefore, $|\cU_v(t+1)|\geq t+1$. $\hfill\square$

	For any node $i$, let $k=n-1$ and designate $i$ as the root. Then, it follows from Claim 2 that $|\widetilde\cU_i(n-1)|\geq n-1$ and $\cR_i\subseteq\cU_i(n-1)$, which proves the result.

\bibliographystyle{agsm}        
\bibliography{mybibf}           

\end{document}